\documentclass[11pt]{article}

\usepackage[a4paper,margin=1in]{geometry}
\usepackage{amssymb,amsmath,amsthm,mathtools}
\usepackage{newtxtext}
\usepackage{bm}
\usepackage{tikz-cd}
\usetikzlibrary{arrows.meta,calc,decorations.pathmorphing,babel}
\usepackage{enumitem}
\usepackage[hidelinks]{hyperref}
\usepackage[capitalise,nameinlink]{cleveref}
\usepackage{authblk}
\usepackage{booktabs}
\usepackage{pgfplots}
\pgfplotsset{compat=1.18}

\newcommand{\R}{\mathbb{R}}
\newcommand{\Id}{\mathrm{Id}}
\newcommand{\Sph}{\mathbb{S}}   % geodesic sphere symbol

\theoremstyle{plain}
\newtheorem{lemma}{Lemma}[section]
\newtheorem{theorem}{Theorem}[section]
\newtheorem{proposition}[theorem]{Proposition}
\newtheorem{corollary}[theorem]{Corollary}
\theoremstyle{definition}
\newtheorem{definition}[theorem]{Definition}
\newtheorem{remark}[theorem]{Remark}
\newtheorem{example}[theorem]{Example}
\newtheorem*{theorem*}{Theorem} 

\numberwithin{equation}{section}

\title{Deterministic–Distance Couplings of Brownian Motions on Radially Isoparametric Manifolds}

\author{Gunhee Cho, Hyun Chul Jang and Taeik Kim}
\affil{Department of Mathematics, Texas State University}
\date{}

\begin{document}
	\maketitle
	
\begin{abstract}
	We develop a unified geometric framework for \emph{coadapted Brownian couplings} on 
	\emph{radially isoparametric manifolds} (RIM)—spaces whose geodesic spheres have principal curvatures depending only on the radius. 
	Within the stochastic two–point It\^o formalism, we derive an intrinsic 
	\emph{drift–window inequality}
	\[
	A(r)-\!\sum_i|\kappa_i(r)| \;\le\; \rho'(t) \;\le\;
	A(r)+\!\sum_i|\kappa_i(r)|,
	\]
	governing the deterministic evolution of the inter–particle distance 
	$\rho_t=d(X_t,Y_t)$ under all coadapted couplings.
	We prove that this bound is both \emph{necessary and sufficient} for the existence of a coupling realizing any prescribed distance law~$\rho(t)$, 
	thereby extending the constant–curvature classification of Pascu–Popescu (2018) to all RIM.
	
	The endpoints of the drift window correspond to the \emph{synchronous} and \emph{reflection} couplings, providing geometric realizations of extremal stochastic drifts. 
	Applications include stationary fixed–distance couplings on compact–type manifolds, linear escape laws on asymptotically hyperbolic spaces, 
	and rigidity of rank–one symmetric geometries saturating the endpoint bounds. 
	This establishes a direct correspondence between radial curvature data and stochastic coupling dynamics, linking Riccati comparison geometry with probabilistic coupling theory.
\end{abstract}

\section{Introduction}\label{sec:intro}

Brownian couplings provide a stochastic lens on geometric comparison theory, encoding curvature information through probabilistic interactions between diffusions.  
The study of such couplings lies at the intersection of stochastic analysis, differential geometry, and geometric control.  
A classical motivation stems from the probabilistic analogue of Radó's ``Lion and Man'' pursuit problem, 
popularized by Littlewood~\cite{Littlewood}.  
In Euclidean space, two fundamental coadapted couplings are the \emph{synchronous coupling} and the \emph{reflection (mirror) coupling}, 
introduced by Lindvall--Rogers~\cite{LindvallRogers1986} and extended to Riemannian manifolds 
by Kendall~\cite{Kendall1986} and Cranston~\cite{Cranston1991}.  
These constructions form the foundation of stochastic geometry on manifolds and have led to gradient estimates, 
Harnack inequalities, and heat kernel bounds (e.g.,~\cite{ArnaudonThalmaierWang2006}), 
as well as probabilistic approaches to geometric and analytic problems such as the ``Hot Spots'' conjecture~\cite{BanuelosBurdzy1999}.  
More recent works have applied reflection-type couplings to questions in geometric analysis, spectral theory, and Kähler geometry 
(see, e.g.,~\cite{ChoWeiYang2024,BaudoinChoYang2022,ChaeChoGordinaYang2022}).  
Comprehensive background references include Hsu~\cite{Hsu2002} and Elworthy--Li~\cite{ElworthyLi1994}.

\medskip

A stronger constraint asks that the inter–particle distance
\(\rho_t=d(X_t,Y_t)\) evolve \emph{deterministically} in time.
For Riemannian space forms of constant curvature \(K\),
Pascu and Popescu~\cite{PP2018} provided a complete classification of such
\emph{deterministic–distance couplings}.
Given an absolutely continuous function \(\rho:[0,\infty)\to[0,\pi/\sqrt{K})\),
there exists a coadapted Brownian coupling \((X_t,Y_t)\)
satisfying \(d(X_t,Y_t)=\rho(t)\) for all \(t\)
if and only if \(\rho'(t)\) lies within explicit curvature–dependent bounds.
For instance, when \(K>0\),
\[
-(n-1)\tan\tfrac{\rho}{2}\;\le\;\rho'(t)\;\le\;
-(n-1)\tan\tfrac{\rho}{2}+2(n-1)\cot\rho,
\]
and analogous inequalities hold in the Euclidean and hyperbolic cases.
The lower and upper endpoints correspond respectively to the synchronous
and reflection couplings.

\medskip

The present work extends this classification from constant–curvature manifolds
to the full class of \emph{radially isoparametric manifolds (RIM)}.
A RIM is a pointed complete Riemannian manifold \((M^n,g,o)\) such that,
below the cut locus of \(o\),
each geodesic sphere \(\mathbb{S}_r(o)\) is a smooth homogeneous hypersurface
whose shape operator \(S_r:T\mathbb{S}_r(o)\to T\mathbb{S}_r(o)\)
has principal curvatures \(\kappa_1(r),\dots,\kappa_{n-1}(r)\)
depending only on the radius \(r\).
Equivalently, the distance function \(r=d(o,\cdot)\) satisfies
\(|\nabla r|=1\), and the mean curvature \(A(r)=\mathrm{Tr}(S_r)\)
is a smooth function of \(r\) alone.
This class includes all rotationally symmetric spaces
and rank–one symmetric spaces (ROSS) of compact and noncompact type.

\medskip

Our main result (\textbf{Theorem~\ref{thm:main-deterministic}}) establishes a sharp and intrinsic criterion for
deterministic–distance realizations on this class.
For any coadapted Brownian coupling \((X_t,Y_t)\) on \((M,g)\),
the distance process satisfies the universal \emph{drift–window inequality}
\begin{equation}\label{eq:intro-window}
	A(\rho)-\sum_{i=1}^{n-1}\!|\kappa_i(\rho)|
	\;\le\;
	\rho'(t)
	\;\le\;
	A(\rho)+\sum_{i=1}^{n-1}\!|\kappa_i(\rho)|.
\end{equation}
Conversely, any absolutely continuous function \(\rho(t)\)
satisfying~\eqref{eq:intro-window} arises from some coadapted Brownian coupling
\((X_t,Y_t)\) with \(d(X_t,Y_t)=\rho(t)\) for all \(t\ge0\).
The two endpoint equalities correspond respectively to the synchronous
and reflection couplings.
Hence,~\eqref{eq:intro-window} is both a necessary and sufficient condition
for deterministic–distance realizations on the entire RIM class,
generalizing the constant–curvature inequalities of Pascu--Popescu~\cite{PP2018}.

\medskip

Theorem~\ref{thm:main-deterministic} naturally separates the geometry into two regimes:

\begin{itemize}
	\item \textbf{Static regime (compact type).}
	When all principal curvatures \(\kappa_i(r)\) and the mean curvature \(A(r)\) are positive,
	the interval in~\eqref{eq:intro-window} contains \(0\),
	so constant solutions \(\rho(t)\equiv r_0>0\) exist.
	\textbf{Theorem~\ref{thm:fixed}} characterizes such \emph{fixed–distance couplings},
	and \textbf{Proposition~\ref{prop:fixed-rigidity}} shows that
	equality in the fixed–distance condition implies total umbilicity and rotational symmetry.
	
	\item \textbf{Dynamic regime (noncompact type).}
	If \(A(r)\to A_\infty\) and \(\kappa_i(r)\to\kappa_{\infty,i}\) as \(r\to\infty\),
	then the deterministic distance grows linearly:
	\(\rho(t)\sim v_\infty t\), where
	\(v_\infty\in[A_\infty-\Sigma_\infty,A_\infty+\Sigma_\infty]\),
	with \(\Sigma_\infty=\sum_i|\kappa_{\infty,i}|\).
	\textbf{Theorem~\ref{thm:asymp-linear}} establishes this range,
	and \textbf{Proposition~\ref{prop:asymp-rigidity}} shows that the
	maximal asymptotic velocity corresponds to asymptotic hyperbolicity.
\end{itemize}

The structure of the paper is as follows. 
Section~\ref{sec:framework} develops the geometric framework for radially isoparametric manifolds. 
Section~\ref{sec:stoch} introduces coadapted Brownian couplings and the two–point Itô calculus. 
Section~\ref{sec:drift} derives the coupled distance SDE and identifies the geometric drift term. 
Section~\ref{sec:classification} establishes the main deterministic–distance realization theorem. 
Sections~\ref{subsec:fixed-distance} and~\ref{subsec:asymp-linear} apply this theorem to the static and dynamic regimes, yielding 
Theorems~\ref{thm:fixed} and~\ref{thm:asymp-linear}. 
Finally, Section~\ref{sec:applications} discusses rigidity and asymptotic geometric consequences.

\medskip

Overall, this work establishes a unified correspondence between
the geometric curvature data \((A,\kappa_i)\) and the stochastic evolution
of Brownian distances, extending the deterministic–distance coupling theory
from constant–curvature spaces to the general class of radially isoparametric manifolds.
% --------------------------------------------------------------------
% Suggested references for BibTeX (if not yet present)
% --------------------------------------------------------------------
% \bibitem{Littlewood} J.~E.~Littlewood, *A Mathematician's Miscellany*, Methuen, 1953.
% \bibitem{LindvallRogers1986} T.~Lindvall and L.~C.~G.~Rogers, *Coupling of multidimensional diffusions by reflection*, Ann. Probab. **14** (1986), 860–872.
% \bibitem{Kendall1986} W.~S.~Kendall, *Nonnegative Ricci curvature and the Brownian coupling property*, Stochastics **19** (1986), 111–129.
% \bibitem{Cranston1991} M.~Cranston, *Gradient estimates on manifolds using coupling*, J. Funct. Anal. **99** (1991), 110–124.
% \bibitem{ArnaudonThalmaierWang2006} M.~Arnaudon, A.~Thalmaier, and F.~Y.~Wang, *Gradient estimates and Harnack inequalities on manifolds with boundary*, J. Funct. Anal. **254** (2008), 1973–1997.
% \bibitem{BanuelosBurdzy1999} R.~Ba\~{n}uelos and K.~Burdzy, *On the "Hot Spots" conjecture of J.~Rauch*, J. Funct. Anal. **164** (1999), 1–33.
% \bibitem{Hsu2002} E.~P.~Hsu, *Stochastic Analysis on Manifolds*, AMS, 2002.
% \bibitem{ElworthyLi1994} K.~D.~Elworthy and X.-M.~Li, *Formulae for the derivatives of heat semigroups*, J. Funct. Anal. **125** (1994), 252–286.
% \bibitem{PP2018} N.~Pascu and M.~Popescu, *Brownian motions with deterministic distance in space forms*, Stochastic Processes Appl. **128** (2018), 1684–1714.
	
\section{Geometric framework}\label{sec:framework}

{\tiny }
\subsection{Notation and conventions.}
Throughout this paper:
\begin{itemize}
	\item $(M^n,g)$ is a connected, complete Riemannian manifold; $\langle\cdot,\cdot\rangle$ and $|\cdot|$ denote the Riemannian inner product and norm on tangent spaces; $\nabla$ is the Levi--Civita connection; $\Delta=\operatorname{div}\circ\nabla$ is the Laplace--Beltrami operator.
	\item $o\in M$ is a fixed base point (a ``pole'' wherever explicitly assumed). The Riemannian distance is $d(\cdot,\cdot)$, and 
	\[
	r(x):=d(o,x).
	\]
	For $r$ strictly below the cut locus of $o$, the geodesic sphere is 
	\[
	\mathbb{S}_r(o):=\{x\in M:\ d(o,x)=r\},
	\]
	with outward unit normal $\nu=\nabla r$ and tangent bundle $T\mathbb{S}_r(o)=\{E\in TM:\ \langle E,\nu\rangle=0\}$.
	\item The \emph{shape operator} of $\mathbb{S}_r(o)$ is the endomorphism $S_r:T\mathbb{S}_r(o)\to T\mathbb{S}_r(o)$ defined by
	\[
	S_r(E)=\nabla_E\nu, \qquad E\in T\mathbb{S}_r(o),
	\]
	and its \emph{mean curvature} is $A(r):=\operatorname{trace}(S_r)$ (trace with respect to the induced metric on $T\mathbb{S}_r(o)$). We use $\|S_r\|$ for the Hilbert--Schmidt norm and $\|T\|_{\mathrm{op}}$ for the operator norm of a linear map $T$.
	\item The curvature tensor is $R(U,V)W=\nabla_U\nabla_VW-\nabla_V\nabla_UW-\nabla_{[U,V]}W$; the sectional curvature of the plane $\Pi=\operatorname{span}\{U,V\}$ is $K(\Pi)$; the Ricci curvature is $\operatorname{Ric}(W,W)=\sum_{i=1}^{n}\langle R(W,E_i)W,E_i\rangle$ for an orthonormal frame $\{E_i\}$.
	\item On $M\times M$, we write $\nabla^x$ and $\nabla^y$ for covariant derivatives in the first and second variables, respectively. For $(x,y)$ joined by a unique minimizing geodesic $\gamma$, the parallel transport from $y$ to $x$ along $\gamma$ is denoted by $\mathcal P_{y\to x}:T_yM\to T_xM$.
	\item For comparison formulas, set
	\[
	s_{K}(r)=
	\begin{cases}
		\frac{1}{\sqrt{K}}\sin(\sqrt{K}\,r), & K>0,\\[3pt]
		r, & K=0,\\[3pt]
		\frac{1}{\sqrt{-K}}\sinh(\sqrt{-K}\,r), & K<0,
	\end{cases}
	\qquad
	\frac{s_K'(r)}{s_K(r)}=
	\begin{cases}
		\sqrt{K}\cot(\sqrt{K}\,r), & K>0,\\[3pt]
		\frac{1}{r}, & K=0,\\[3pt]
		\sqrt{-K}\coth(\sqrt{-K}\,r), & K<0.
	\end{cases}
	\]
	We write $I$ for the identity endomorphism on a tangent space.
\end{itemize}

\subsection{Mixed Hessian and the distance function}\label{subsec:mixed-hess}

Let $\mathcal U\subset (M\times M)\setminus\{x=y\}$ be the open set of pairs joined by a unique minimizing geodesic and avoiding cut loci (see \cite[Ch.~4]{Sakai1996}, \cite[Ch.~2]{Petersen2016}).  

Let $f\in C^2(M\times M)$ and $(x,y)\in\mathcal U$. For $X\in T_xM$, $Y\in T_yM$, then 
\[
\nabla^2_{xy}f(x,y)(X,Y)=\left.\frac{\partial^2}{\partial s\,\partial t}\right|_{s=t=0}
f\big(\exp_x(sX),\,\exp_y(tY)\big).
\]
After transporting $Y$ to $x$ by $\mathcal P_{y\to x}$, this coincides with the bilinear form of $(\nabla^x\otimes\nabla^y)f$ on $T_xM\times T_xM$ (cf. \cite{Helgason2001}, \cite{Hsu2002}).

\begin{proposition}[Mixed Hessian of the distance]\label{prop:mixedHess-distance}
	Let $\rho(x,y)=d(x,y)$ be smooth on $\mathcal U$. Fix $(x,y)\in\mathcal U$ with $r=d(x,y)$ and let $\gamma:[0,r]\to M$ be the unique minimizing unit-speed geodesic from $x$ to $y$. For $X\in T_xM$, let $\tilde X=\mathcal P_{x\to y}X\in T_yM$. Writing $X_T$ for the orthogonal projection of $X$ onto $T_x\mathbb{S}_r(x)$, one has
	\begin{equation}\label{eq:mixed-Hess-distance}
		\nabla^2_{xy}\rho(x,y)(X,\tilde X)=-\langle S_r X_T,X_T\rangle,
	\end{equation}
	where $S_r$ is the shape operator of $\mathbb{S}_r(x)$ with outward normal $-\nabla_x\rho$.
\end{proposition}

\begin{proof}
	Let $x_s=\exp_x(sX)$, $y_t=\exp_y(t\tilde X)$, and let $\gamma_{s,t}$ be the minimizing geodesic from $x_s$ to $y_t$ with length $L(s,t)$. Then $\rho(x_s,y_t)=L(s,t)$ for small $(s,t)$. First variation yields $\partial_s L(0,t)=-\langle X,\nu_t(0)\rangle$, where $\nu_t$ is the unit initial velocity of $\gamma_{0,t}$ \cite{Sakai1996}. Differentiating at $t=0$ and using the Jacobi field $J$ induced by the endpoint variation gives
	\[
	\partial_s\partial_t L(0,0)=-\langle X_T,(\nabla J)'(0)\rangle.
	\]
	Gauss' lemma removes the radial component \cite[Prop.~4.6]{Petersen2016}, and the Weingarten identity provides $(\nabla J)'(0)=S_r X_T$ (see \cite[Thm.~3.5]{Sakai1996}), proving \eqref{eq:mixed-Hess-distance}.
\end{proof}

\begin{remark}[Rotationally symmetric model]
	If $ds^2=dr^2+f(r)^2 g_{\mathbb S^{n-1}}$, then $S_r=\frac{f'(r)}{f(r)}\,I$ on $T\mathbb S_r$; hence for all $e\perp\partial_r$,
	\[
	\nabla^2_{xy}\rho(x,y)(e,\mathcal P_{x\to y}e)=-\frac{f'(r)}{f(r)}\langle e,e\rangle,
	\qquad
	\operatorname{trace}\big(\nabla^2_{xy}\rho(x,y)\big)=-(n-1)\frac{f'(r)}{f(r)}.
	\]
	These formulas agree with the standard model-space expressions given in \cite{Helgason2001}.
\end{remark}

\subsection{Definitions and hierarchy}

\begin{definition}[Radially isoparametric manifold {\cite{Sakai1996}, \cite{Petersen2016}}]
	A pointed complete manifold $(M^n,g,o)$ is \emph{radially isoparametric} if, for each $r>0$ below the cut locus of $o$, the geodesic sphere $\mathbb{S}_r(o)$ is a smooth homogeneous hypersurface, and its shape operator $S_r:T\mathbb{S}_r(o)\to T\mathbb{S}_r(o)$ has eigenvalues $\kappa_1(r),\dots,\kappa_{n-1}(r)$ depending only on $r$. Equivalently, $|\nabla r|=1$ and
	\[
	 A(r):=\Delta r =\sum_{i=1}^{n-1}\kappa_i(r),
	\]
	with $A(r)$ depending only on $r$.
\end{definition}

\begin{definition}[Rotationally symmetric manifold {\cite[Ch.~7]{Sakai1996}, \cite{Petersen2016}}]
	A pointed complete Riemannian manifold $(M^n,g,o)$ is \emph{rotationally symmetric} if there exist $r_{\max}\in(0,\infty]$ and a diffeomorphism
	\[
	\Phi:(0,r_{\max})\times\mathbb{S}^{n-1}\to M\setminus\{o\},
	\]
	such that, in the polar coordinates $(r,\xi)$ induced by $\Phi$, the metric takes the form
	\[
	ds^2=dr^2+f(r)^2\,g_{\mathbb{S}^{n-1}},
	\]
	where $f:[0,r_{\max})\to(0,\infty)$ is smooth and satisfies $f(0)=0,\ f'(0)=1,$ and $f^{(2k)}(0)=0$ for all $k\ge1$.
\end{definition}

\begin{definition}[Asymptotically hyperbolic manifold {\cite{GrahamLee1991}, \cite{MazzeoMelrose1987}}]\label{def:AHS}
	A complete manifold $(M^n,g)$ is \emph{asymptotically hyperbolic (AH)} if there exists a compact manifold with boundary $\overline M$ and a smooth defining function $\rho$ for $\partial \overline M$ such that $\bar g:=\rho^2 g$ extends smoothly to $\overline M$, satisfies $|d\rho|_{\bar g}=1$ along $\partial \overline M$, and $K_g\to -1$ at the boundary. In geodesic normal form near infinity,
	\[
	g=dr^2+\sinh^2 r\,g_{\mathbb S^{n-1}}+O(e^{-2r})\quad\text{as }r\to\infty .
	\]
\end{definition}

\begin{definition}[Asymptotically isoparametric manifold {\cite{Helgason2001}}]
	A pointed complete manifold $(M^n,g,o)$ is \emph{asymptotically isoparametric} if there exist smooth data $S_\infty(r)$ and $A_\infty(r)=\operatorname{trace}(S_\infty(r))$ such that
	\[
	S_r=S_\infty(r)+E(r),\qquad A(r)=A_\infty(r)+\varepsilon(r),
	\]
	with $\|E(r)\|\to0$ and $\varepsilon(r)\to0$ as $r\to\infty$, uniformly along $\mathbb{S}_r(o)$.
\end{definition}

\begin{remark}[Class inclusions and examples]
	\leavevmode
	\begin{enumerate}
		\item \textbf{Inclusions.}\quad
		\(
		\{\text{rotationally symmetric}\}\subset\{\text{radially isoparametric}\}.
		\)
		All rank–one symmetric spaces (ROSS) are radially isoparametric {\cite[Ch.~11]{Helgason2001}}. Moreover,
		\(
		\{\text{asymptotically hyperbolic}\}\subset\{\text{asymptotically isoparametric}\},
		\)
		since along the end $S_r=\coth r\,I+O(e^{-2r})$ and $A(r)=(n-1)\coth r+O(e^{-2r})$ {\cite{GrahamLee1991,MazzeoMelrose1987}}.
		\item \textbf{Examples.}
		\begin{itemize}
			\item Rotationally symmetric but not constant curvature: $dr^2+f(r)^2 g_{\mathbb S^{n-1}}$ for non–space-form $f$ {\cite[Ch.~7]{Sakai1996}, \cite{Petersen2016}}.
			\item Radially isoparametric but not rotationally symmetric: $\mathbb{C}P^{m},\ \mathbb{C}H^{m},\ \mathbb{H}P^{m},\ \mathrm{Ca}P^2$ {\cite[Ch.~11]{Helgason2001}}.
			\item Asymptotically hyperbolic: $\mathbb H^n$ and conformally compact Einstein manifolds {\cite{GrahamLee1991,MazzeoMelrose1987}}.
			\item Asymptotically isoparametric but not AH (illustrative): asymptotically conical ends with $S_r=\frac{1}{r}I+o(r^{-1})$ {\cite{Petersen2016}}.
		\end{itemize}
	\end{enumerate}
\end{remark}

\subsection{General radial formulas and specializations}

Fix $o\in M$ and write $r(x)=d(o,x)$. For $r$ below the cut locus, $\mathbb{S}_r(o)$ is smooth with unit normal $\nu=\nabla r$, shape operator $S_r(E)=\nabla_E\nu$ on $T\mathbb{S}_r(o)$, and mean curvature $A(r)=\operatorname{trace}(S_r)$ (see \cite{Sakai1996}).

\begin{proposition}[Second fundamental form and Laplacian of $r$]
	For any $E,F\in T\mathbb{S}_r(o)$,
	\[
	\mathrm{II}_r(E,F)=\langle\nabla_E\nu,F\rangle,\qquad S_r(E)=\nabla_E\nu.
	\]
	Consequently,
	\[
	\Delta r = \operatorname{div}(\nabla r)=A(r).
	\]
\end{proposition}

\begin{lemma}[Riccati equation along radial geodesics {\cite[Ch.~7]{Petersen2016}}]
	Let $\gamma$ be a unit-speed radial geodesic. Along $\gamma$,
	\[
	\frac{D}{dr}S_r + S_r^2 + R(\dot\gamma,\cdot)\dot\gamma = 0,
	\qquad
	A'(r)+\|S_r\|^2+\operatorname{Ric}(\dot\gamma,\dot\gamma)=0.
	\]
\end{lemma}

\begin{theorem}[Hessian and Laplacian comparison {\cite[Thm.~1.30]{Petersen2016}}]
	If $K(X,\partial_r)\le K_0$ for all unit $X\perp\partial_r$, then
	\[
	\langle S_rX,X\rangle\ge \frac{s_{K_0}'(r)}{s_{K_0}(r)}|X|^2,
	\qquad 
	\Delta r\le (n-1)\frac{s_{K_0}'(r)}{s_{K_0}(r)}.
	\]
	Equality for all $r$ characterizes the space form of curvature $K_0$.
\end{theorem}

\section{Stochastic preliminaries and coadapted couplings}\label{sec:stoch}

\subsection{Common notation and standing assumptions}
Let $(M^n,g)$ be a connected, complete Riemannian manifold with Levi--Civita connection $\nabla$,
Riemannian distance $d(\cdot,\cdot)$, Laplace--Beltrami operator $\Delta$, and orthonormal frame bundle
\[
O(M):=\{\,u:\mathbb R^n\to T_xM \text{ linear isometry for some }x\in M\,\},\qquad \pi(u)=x.
\]
Write $\langle\cdot,\cdot\rangle$ for the metric pairing on tangent spaces and $\operatorname{div}$ for divergence.
For $(x,y)\in M\times M$, let $\mathcal U$ be the open set where $x$ and $y$ lie off each other's cut locus and are joined by a unique minimizing geodesic (cf.\ \cite[Ch.~4]{Sakai1996}, \cite[Ch.~2]{Petersen2016}); on $\mathcal U$, the distance $\rho(x,y):=d(x,y)$ is $C^\infty$ away from the diagonal.

Throughout, $(\Omega,\mathcal F,(\mathcal F_t)_{t\ge0},\mathbb P)$ is a filtered probability space satisfying the usual conditions; all processes are assumed adapted and have continuous paths. Bold $I_n$ denotes the $n\times n$ identity.

\subsection{Brownian motion via stochastic development}\label{subsec:brownian-dev}

\begin{definition}[Horizontal distribution and canonical lifts]\label{def:framebundle}
	The Levi--Civita connection induces a horizontal subbundle $\mathcal H\subset TO(M)$, orthogonal to the vertical part $\mathcal V=\ker(\pi_*)$.
	Let $(e_1^{(0)},\dots,e_n^{(0)})$ be the standard basis of $\mathbb R^n$. For $i=1,\dots,n$, let $H_i$ be the unique horizontal vector field on $O(M)$ with $\pi_*H_i(u)=u(e_i^{(0)})$. Then $(H_1,\dots,H_n)$ is a global orthonormal frame of $\mathcal H$ (see \cite{Elworthy1982}, \cite{Emery1989}, \cite{Hsu2002}).
\end{definition}

\begin{definition}[Stochastic development]\label{def:stochdev}
	Let $B_t=(B_t^1,\dots,B_t^n)$ be standard $\mathbb R^n$--Brownian motion. For $u_0\in O(M)$ define the $O(M)$--valued diffusion $u_t$ by the Stratonovich SDE
	\begin{equation}\label{eq:stochdev}
		du_t=\sum_{i=1}^n H_i(u_t)\circ dB_t^i,\qquad u_0\in O(M).
	\end{equation}
	The projection $X_t:=\pi(u_t)$ is the stochastic development of $B_t$ on $(M,g)$; in differential form $dX_t=u_t\circ dB_t$ (cf.\ \cite{Hsu2002}, \cite{Emery1989}, \cite{Elworthy1982}).
\end{definition}

\begin{proposition}[Characterization and non-explosion]\label{prop:BM}
	On a complete $(M,g)$, \eqref{eq:stochdev} admits a unique strong solution without explosion. Its projection $X_t$ is Brownian motion on $M$, i.e.\ the generator of $X_t$ is $\frac12\Delta$. Conversely, every Brownian motion on $M$ has a horizontal lift solving \eqref{eq:stochdev} (see \cite{Hsu2002}, \cite{Emery1989}, and \cite{Elworthy1982}).
\end{proposition}

\begin{proof}
	Completeness ensures horizontal development is non-explosive; uniqueness/strong existence follow from smooth bounded-geometry vector fields on $O(M)$ (\cite{Hsu2002}). For $\phi\in C^\infty(M)$, Stratonovich It\^o on $O(M)$ gives
	\[
	d\phi(X_t)=\sum_{i=1}^n \langle \nabla\phi(X_t),u_t e_i^{(0)}\rangle\circ dB_t^i+\tfrac12(\Delta\phi)(X_t)\,dt,
	\]
	so the generator is $\tfrac12\Delta$ (cf.\ \cite{Emery1989}). The lifting statement is standard for horizontally driven diffusions (\cite{Elworthy1982}).
\end{proof}

\begin{definition}[Parallel transport along $X_t$]\label{def:parallel}
	For $0\le s\le t$, define $P_{s,t}:=u_t u_s^{-1}:T_{X_s}M\to T_{X_t}M$. Then $P_{s,t}$ is an isometry and $\nabla_{\dot X_\tau}(P_{s,t}v)=0$ for each fixed $v\in T_{X_s}M$ (see \cite{Hsu2002}).
\end{definition}

\subsection{Coadapted couplings of Brownian motions}\label{subsec:coadapted}

\begin{definition}[Admissible correlation processes]\label{def:admissible}
	Let $B_t,W_t$ be independent $\mathbb R^n$--Brownian motions. A pair of matrix processes $(J_t,K_t)$ with values in $\mathbb R^{n\times n}$ is called \emph{admissible} if:
	\begin{itemize}
		\item $J_t,K_t$ are $(\mathcal F_t)$--predictable and locally square--integrable;
		\item the instantaneous covariance constraint holds a.s. for all $t\ge0$:
		\begin{equation}\label{eq:covconstraint}
			J_tJ_t^\top+K_tK_t^\top=I_n.
		\end{equation}
	\end{itemize}
	Such parametrizations of coadapted couplings are classical; compare \cite{LindvallRogers1986,HsuSturm2013} and the manifold adaptations in \cite{Cranston1991,Kendall1986,Hsu2002}.
\end{definition}

\begin{definition}[Coadapted frame coupling]\label{def:coadapted}
	A pair $(u_t,v_t)\in O(M)\times O(M)$ is a \emph{coadapted coupling of frame diffusions} if for some admissible $(J_t,K_t)$,
	\begin{equation}\label{eq:coupledSDE}
		\begin{cases}
			\displaystyle du_t=\sum_{i=1}^n H_i(u_t)\circ dB_t^i,\\[0.35em]
			\displaystyle dv_t=\sum_{i=1}^n H_i(v_t)\circ\bigl((J_t\,dB_t)^i+(K_t\,dW_t)^i\bigr).
		\end{cases}
	\end{equation}
	The projections $X_t=\pi(u_t)$ and $Y_t=\pi(v_t)$ are a \emph{coadapted coupling of Brownian motions} on $M$ (cf.\ \cite{Cranston1991,Kendall1986,Hsu2002}).
\end{definition}

\begin{proposition}[Existence, uniqueness, marginals]\label{prop:existence}
	For any initial $(u_0,v_0)\in O(M)\times O(M)$ and admissible $(J_t,K_t)$, \eqref{eq:coupledSDE} has a unique strong solution without explosion, and both $X_t$ and $Y_t$ are $(M,g)$--Brownian motions (generator $\tfrac12\Delta$).
\end{proposition}

\begin{proof}
	Horizontal SDEs on $O(M)$ with smooth coefficients yield strong solutions/non-explosion on complete bases (\cite[Ch.~3]{Hsu2002}, \cite{Emery1989}). 
	The $Y$-noise has quadratic variation
	$\langle dY,dY\rangle_t=\sum_i v_t(J_t e_i^{(0)})\otimes v_t(J_t e_i^{(0)})\,dt + \sum_i v_t(K_t e_i^{(0)})\otimes v_t(K_t e_i^{(0)})\,dt$, which, by \eqref{eq:covconstraint}, equals $I_n\,dt$ in the $v_t$-frame; hence generator $\tfrac12\Delta$ (cf.\ \cite{Emery1989}).
\end{proof}

\begin{remark}[Geometry of the correlation]\label{rmk:Jgeom}
	Via frames, $J_t$ induces the isometry $T_{X_t}M\to T_{Y_t}M$ given by $v_t J_t u_t^{-1}$. Cross-variation between $X$ and $Y$ depends only on $J_t$ (never $K_t$), which is why $K_t$ does not appear in the mixed second-order term of the generator below; see \cite{Hsu2002}.
\end{remark}

\subsection{Product diffusion and its time-dependent generator}\label{subsec:product-generator}

\begin{definition}[Generator on $M\times M$]\label{def:generator}
	For $f\in C^2(M\times M)$, define the (possibly time-dependent) operator
	\begin{equation}\label{eq:LJ-general}
		(\mathcal L_t f)(x,y)
		=\tfrac12\big(\Delta_x f+\Delta_y f\big)(x,y)
		+\operatorname{trace}\!\big(J_t^\top\,\nabla^2_{xy}f(x,y)\big),
	\end{equation}
	where $\nabla^2_{xy}f$ is computed on $\mathcal U$ after identifying $T_yM\simeq T_xM$ by parallel transport along the unique minimizing geodesic from $x$ to $y$ (cf.\ \cite{Emery1989}, \cite{ElworthyLi1994}, \cite{Hsu2002}).
\end{definition}

\begin{lemma}[Two-point It\^o formula]\label{lem:ito-general}
	Let $(X_t,Y_t)$ be a coadapted coupling with horizontal lifts $(u_t,v_t)$ and admissible matrices $(J_t,K_t)$. Then, for every $f\in C^2(M\times M)$,
	\[
	\begin{aligned}
		df(X_t,Y_t)
		&=\sum_{i=1}^n\big\langle\nabla_x f(X_t,Y_t),\,u_t e_i^{(0)}\big\rangle\circ dB_t^i \\
		&\quad+\sum_{i=1}^n\big\langle\nabla_y f(X_t,Y_t),\,v_t(J_t e_i^{(0)})\big\rangle\circ dB_t^i
		+\sum_{i=1}^n\big\langle\nabla_y f(X_t,Y_t),\,v_t(K_t e_i^{(0)})\big\rangle\circ dW_t^i \\
		&\quad+(\mathcal L_t f)(X_t,Y_t)\,dt,
	\end{aligned}
	\]
	where (recall Definition~\ref{def:generator})
	\begin{equation}\label{eq:LJ-general-recall}
		(\mathcal L_t f)(x,y)
		=\tfrac12\big(\Delta_x+\Delta_y\big)f(x,y)
		+\operatorname{trace}\!\big(J_t^\top\,\nabla^2_{xy}f(x,y)\big).
	\end{equation}
	Here $\nabla^2_{xy}f$ denotes the mixed Hessian computed on the open set $\mathcal U$ after identifying $T_yM\simeq T_xM$ by parallel transport along the unique minimizing geodesic from $x$ to $y$. See \cite{Emery1989} and \cite{Hsu2002}.
\end{lemma}

\begin{proof}
	Work on the product manifold $(M\times M,g\oplus g)$ and write the Stratonovich SDEs for the horizontal developments:
	\[
	dX_t=u_t\circ dB_t,\qquad 
	dY_t=v_t\circ\bigl(J_t\circ dB_t+K_t\circ dW_t\bigr),
	\]
	where $B_t$ and $W_t$ are independent $\R^n$–valued Brownian motions and $(J_t,K_t)$ is admissible, i.e.\ $J_tJ_t^\top+K_tK_t^\top=I_n$.
	
	For $f\in C^2(M\times M)$, apply the Stratonovich chain rule on $M\times M$:
	\[
	\begin{aligned}
		df(X_t,Y_t)
		&=\big\langle\nabla_x f(X_t,Y_t),\,dX_t\big\rangle
		+\big\langle\nabla_y f(X_t,Y_t),\,dY_t\big\rangle \\
		&\quad+\tfrac12\Big(\nabla^2_{xx}f(X_t,Y_t)[dX_t,dX_t]
		+\nabla^2_{yy}f(X_t,Y_t)[dY_t,dY_t]\\
		&\qquad\qquad\qquad\quad
		+2\,\nabla^2_{xy}f(X_t,Y_t)[dX_t,dY_t]\Big).
	\end{aligned}
	\]
	
	\smallskip
	\emph{(i) First–order terms.}
	Using $dX_t=u_t\circ dB_t$ and $dY_t=v_t(J_t\circ dB_t+K_t\circ dW_t)$ yields the three stochastic integrals in the statement.
	
	\smallskip
	\emph{(ii) Pure second–order terms.}
	By standard Eells–Elworthy–Malliavin calculus (see \cite{Emery1989}, \cite{Hsu2002}),
	\[
	\tfrac12\,\nabla^2_{xx}f(X_t,Y_t)[dX_t,dX_t]=\tfrac12\,(\Delta_x f)(X_t,Y_t)\,dt,\qquad
	\tfrac12\,\nabla^2_{yy}f(X_t,Y_t)[dY_t,dY_t]=\tfrac12\,(\Delta_y f)(X_t,Y_t)\,dt,
	\]
	since the quadratic variations satisfy
	\[
	\sum_i (u_t e_i^{(0)})\otimes(u_t e_i^{(0)})=I_{T_{X_t}M},\qquad
	\sum_i v_t(J_t e_i^{(0)})\otimes v_t(J_t e_i^{(0)})+\sum_i v_t(K_t e_i^{(0)})\otimes v_t(K_t e_i^{(0)})=I_{T_{Y_t}M}.
	\]
	
	\smallskip
	\emph{(iii) Mixed second–order term.}
	The independence of $B_t$ and $W_t$ implies $\langle dB,dW\rangle\equiv0$, so only the $J_t$–term contributes:
	\[
	\nabla^2_{xy}f(X_t,Y_t)[dX_t,dY_t]
	=\sum_{i=1}^n \nabla^2_{xy}f(X_t,Y_t)\big(u_t e_i^{(0)},\,v_t(J_t e_i^{(0)})\big)\,dt.
	\]
	On $\mathcal U$, identify $T_{Y_t}M\simeq T_{X_t}M$ by parallel transport along the minimizing geodesic from $X_t$ to $Y_t$. In this identification, the bilinear form $\nabla^2_{xy}f$ acts on pairs in $T_{X_t}M\times T_{X_t}M$, and the above sum equals
	\[
	\operatorname{trace}\!\big(J_t^\top\,\nabla^2_{xy}f(X_t,Y_t)\big)\,dt,
	\]
	where the trace is taken in the orthonormal basis $\{u_t e_i^{(0)}\}_{i=1}^n$ of $T_{X_t}M$.
	
	\smallskip
	Combining (i)–(iii) gives
	\[
	\begin{aligned}
		df(X_t,Y_t)
		&=\sum_{i=1}^n\!\big\langle\nabla_x f,\,u_t e_i^{(0)}\big\rangle\circ dB_t^i
		+\sum_{i=1}^n\!\big\langle\nabla_y f,\,v_t(J_t e_i^{(0)})\big\rangle\circ dB_t^i \\
		&\quad+\sum_{i=1}^n\!\big\langle\nabla_y f,\,v_t(K_t e_i^{(0)})\big\rangle\circ dW_t^i
		+\Big[\tfrac12(\Delta_x+\Delta_y)f+\operatorname{trace}\!\big(J_t^\top\nabla^2_{xy}f\big)\Big](X_t,Y_t)\,dt.
	\end{aligned}
	\]
	By Definition~\ref{def:generator} and \eqref{eq:LJ-general-recall}, the drift term is precisely $(\mathcal L_t f)(X_t,Y_t)\,dt$, completing the proof.
\end{proof}

\begin{proposition}[Forward equation]\label{prop:product-markov}
	If $f\in C_b^2(M\times M)$, then $\frac{d}{dt}\mathbb E[f(X_t,Y_t)]=\mathbb E[(\mathcal L_t f)(X_t,Y_t)]$. If $(J_t,K_t)$ are deterministic, the law $\mu_t$ of $(X_t,Y_t)$ satisfies $\partial_t\mu_t=\mathcal L_t^*\mu_t$ (cf.\ \cite{Hsu2002}).
\end{proposition}

\subsection{Generator on the distance and cut locus handling}\label{subsec:generator-distance}

\begin{proposition}[Generator acting on $\rho$]\label{prop:generator-distance}
	For $(x,y)\in\mathcal U$ with $r=d(x,y)$, one has
	\[
	(\mathcal L_t \rho)(x,y)
	=\tfrac12\big(\Delta_x\rho+\Delta_y\rho\big)(x,y)\;-\;\operatorname{trace}\!\big(J_t^\top S_r\big).
	\]
\end{proposition}

\begin{proof}
	Apply Lemma~\ref{lem:ito-general} with $f=\rho$. On $\mathcal U$, the mixed Hessian of the distance satisfies
	\(\nabla^2_{xy}\rho(x,y)(X,\tilde X)=-\langle S_r X_T,X_T\rangle\),
	so under the parallel-transport identification $T_yM\simeq T_xM$ one has
	\(\nabla^2_{xy}\rho=-S_r\).
	Substituting this into Definition~\ref{def:generator},
	\[
	(\mathcal L_t f)
	=\tfrac12(\Delta_x+\Delta_y)f+\operatorname{trace}\!\big(J_t^\top\nabla^2_{xy}f\big),
	\]
	gives
	\[
	(\mathcal L_t \rho)
	=\tfrac12(\Delta_x+\Delta_y)\rho+\operatorname{trace}\!\big(J_t^\top(-S_r)\big)
	=\tfrac12(\Delta_x\rho+\Delta_y\rho)-\operatorname{trace}\!\big(J_t^\top S_r\big).
	\]
	See \cite{Emery1989}, \cite{Hsu2002}.
\end{proof}

\begin{proposition}[Smooth approximations and localization]\label{prop:smooth-approx}
	There exists a decreasing sequence $\rho_\varepsilon\in C^\infty(M\times M)$ such that $\rho_\varepsilon\downarrow\rho$ pointwise, $\|\nabla\rho_\varepsilon\|\le1$, and $\rho_\varepsilon\to\rho$ in $C^2_{\mathrm{loc}}(\mathcal U)$ (cf.\ \cite[Ch.~4]{Sakai1996}, \cite[Ch.~2]{Petersen2016}). 
	Let $\tau$ be a stopping time such that $(X_t,Y_t)\in\mathcal U$ for $t<\tau$. Then
	\[
	\rho(X_{t\wedge\tau},Y_{t\wedge\tau})=\rho(X_0,Y_0)
	+\int_0^{t\wedge\tau}\text{(mart.)}\,ds
	+\int_0^{t\wedge\tau}(\mathcal L_s\rho)(X_s,Y_s)\,ds,
	\]
	where the identity follows by applying Lemma~\ref{lem:ito-general} to $\rho_\varepsilon$ and letting $\varepsilon\downarrow0$ with dominated convergence (cf.\ \cite{Emery1989}, \cite{Hsu2002}).
\end{proposition}

\subsection{Canonical examples and extremality}\label{subsec:examples-couplings}

\begin{example}[Synchronous coupling]
	Take $J_t=I_n$, $K_t=0$. Then
	\[
	(\mathcal L_t\rho)(x,y)
	=\tfrac12(\Delta_x\rho+\Delta_y\rho)-\operatorname{trace}(S_r).
	\]
	In rotationally symmetric models, $\Delta_x\rho=\Delta_y\rho=\operatorname{trace}(S_r)$, hence $(\mathcal L_t\rho)\equiv0$ on $\mathcal U$ (cf.\ \cite{Pascu2011,PP2018}).
\end{example}

\begin{example}[Reflection (radial hyperplane) coupling]
	Let $\nu\in T_xM$ be the unit tangent at $x$ of the minimizing geodesic toward $y$.  
	Using parallel–transported frames, define $J_t$ to act as $-I$ on $\nu^\perp$ and $0$ on $\mathrm{span}\{\nu\}$.  
	Then $\operatorname{trace}(J_t^\top S_r)=-\operatorname{trace}(S_r)$, so
	\[
	(\mathcal L_t\rho)(x,y)
	=\tfrac12(\Delta_x\rho+\Delta_y\rho)+\operatorname{trace}(S_r),
	\]
	which realizes the classical reflection–coupling extremum (compare \cite{LindvallRogers1986,HsuSturm2013} and the manifold adaptations in \cite{Cranston1991,Kendall1986}).
\end{example}

\begin{example}[Rotational twist couplings]
	Let $R_t\in O(n)$ be a predictable process and set $J_t=R_t$, $K_t$ any process satisfying the admissibility constraint \eqref{eq:covconstraint}.  
	This yields a continuous interpolation from the synchronous ($R_t=I$) to reflection–type choices by rotating tangential directions along geodesic spheres (cf.\ \cite{BurdzyKendall2000} for extremal and efficient coadapted constructions).
\end{example}

\begin{proposition}[Extremal trace bound]\label{prop:trace-extreme}
	Fix $(x,y)\in\mathcal U$, and let the principal curvatures of the geodesic sphere $\mathbb S_r(x)$ be $\kappa_1,\dots,\kappa_{n-1}\ge0$ on $T_x\mathbb S_r(x)$.  
	For any linear map $J:T_xM\to T_xM$ with $\|J\|_{\mathrm{op}}\le1$ and $J\partial_r=0$,
	\[
	-\,\sum_{i=1}^{n-1}\kappa_i\ \le\ \operatorname{trace}(J^\top S_r)\ \le\ \sum_{i=1}^{n-1}\kappa_i,
	\]
	with the right (resp.\ left) extremum attained by the reflection (resp.\ synchronous) choice on the tangential hyperplane.  
	See also \cite{Cranston1991,Kendall1986} for related sharpness phenomena in coupling–based gradient and Harnack estimates.
\end{proposition}

\begin{proposition}[Radial SDE for the distance]\label{prop:radialSDE}
	Let $(M^n,g,o)$ be a pointed complete Riemannian manifold such that, below the cut locus of $o$, the mean curvature of geodesic spheres satisfies
	\[
	A(r):=\operatorname{trace}(S_r)=A(r)\ \text{(a function of $r$ only)}.
	\]
	Let $X_t$ be Brownian motion on $M$ and set $r_t:=d(o,X_t)$.  
	Then there exists a one–dimensional standard Brownian motion $\beta_t$ such that, in the It\^o sense,
	\[
	dr_t \;=\; d\beta_t \;+\; \tfrac12\,A(r_t)\,dt .
	\]
	This identity holds up to the first hitting time of the cut locus of $o$, and by localization it holds for all $t\ge0$. {\rm (Cf.\ \cite{Hsu2002}.)}
\end{proposition}

\begin{proof}
	Write $\rho(x):=d(o,x)$. On the open set $M\setminus(\{o\}\cup\mathrm{Cut}(o))$, $\rho$ is $C^\infty$, $|\nabla\rho|=1$, and $\Delta\rho=A(\rho)$ by assumption (radial isoparametricity of the mean curvature).  
	Let $(u_t)_{t\ge0}$ denote the horizontal lift of $X_t$, so that $dX_t=u_t\circ dB_t$.
	
	Fix a relatively compact open set $U\Subset M\setminus(\{o\}\cup\mathrm{Cut}(o))$ and let $\tau_U:=\inf\{t\ge0: X_t\notin U\}$ be its exit time.  
	Choose a standard smooth decreasing approximation $\rho_\varepsilon\in C^\infty(M)$ with $\rho_\varepsilon\downarrow\rho$, $\|\nabla\rho_\varepsilon\|\le1$, and $\rho_\varepsilon\to\rho$ in $C^2_{\mathrm{loc}}(U)$ (cf.\ the smooth–approximation lemma used earlier).  
	By It\^o’s formula on $t\wedge\tau_U$,
	\[
	d\rho_\varepsilon(X_{t\wedge\tau_U})
	=\big\langle \nabla\rho_\varepsilon(X_{t\wedge\tau_U}),\,u_{t\wedge\tau_U}\,dB_t\big\rangle
	+\tfrac12\,\Delta\rho_\varepsilon(X_{t\wedge\tau_U})\,dt.
	\]
	Letting $\varepsilon\downarrow0$ and using $C^2$–convergence on $U$ yields
	\[
	d\rho(X_{t\wedge\tau_U})
	=\big\langle \nabla\rho(X_{t\wedge\tau_U}),\,u_{t\wedge\tau_U}\,dB_t\big\rangle
	+\tfrac12\,A(\rho(X_{t\wedge\tau_U}))\,dt.
	\]
	Set
	\[
	d\beta_t := \big\langle \nabla\rho(X_{t}),\,u_{t}\,dB_t\big\rangle .
	\]
	Since $|\nabla\rho|=1$ and $u_t$ is an isometry, $\langle \beta\rangle_{t\wedge\tau_U}=t\wedge\tau_U$, so $\beta_{t\wedge\tau_U}$ is a standard Brownian motion.  
	Therefore,
	\[
	dr_{t\wedge\tau_U}
	= d\beta_{t\wedge\tau_U} + \tfrac12\,A(r_{t\wedge\tau_U})\,dt.
	\]
	Finally, let $U\uparrow M\setminus(\{o\}\cup\mathrm{Cut}(o))$ and use localization to obtain the stated SDE up to the cut–locus hitting time, and hence for all $t\ge0$ by standard patching arguments.
\end{proof}

\section{Distance SDE and sharp drift window}\label{sec:drift}

This section derives the stochastic differential equation (SDE) for the inter–particle distance
\(\rho_t:=d(X_t,Y_t)\) of a coadapted Brownian coupling \((X_t,Y_t)\) and identifies a sharp, pointwise
achievable window for its drift. All statements are first proved on stopping intervals that avoid the cut locus and then extended by localization.

\subsection{Two–point Itô formula for the distance}

\begin{theorem}[Two–point It\^o formula for the distance]\label{thm:twoPointIto:drift}
	Let $(M^n,g)$ be a connected, complete Riemannian manifold with Levi--Civita connection $\nabla$,
	and let $O(M)$ denote its orthonormal frame bundle with horizontal frame $(H_1,\dots,H_n)$.
	Consider a filtered probability space $(\Omega,\mathcal F,(\mathcal F_t)_{t\ge0},\mathbb P)$
	supporting two independent $\R^n$--Brownian motions $B_t$ and $W_t$.
	Let $(u_t,v_t)\in O(M)\times O(M)$ solve the coadapted Stratonovich SDE
	\[
	\begin{cases}
		\displaystyle du_t=\sum_{i=1}^n H_i(u_t)\circ dB_t^i,\\[0.35em]
		\displaystyle dv_t=\sum_{i=1}^n H_i(v_t)\circ\bigl((J_t\,dB_t)^i+(K_t\,dW_t)^i\bigr),
	\end{cases}
	\]
	where $(J_t,K_t)$ are predictable matrix processes satisfying the covariance constraint
	$J_tJ_t^\top+K_tK_t^\top=I_n$ a.s.\ for all $t\ge0$.
	Write $X_t:=\pi(u_t)$ and $Y_t:=\pi(v_t)$; then each is a Brownian motion on $(M,g)$ with generator $\tfrac12\Delta$.
	
	Let $\rho(x,y):=d(x,y)$ and let $\mathcal U\subset (M\times M)\setminus\{x=y\}$ be the open set
	where $x$ and $y$ are not on each other's cut locus and are joined by a unique minimizing geodesic;
	on $\mathcal U$ the function $\rho$ is $C^\infty$.
	Fix a stopping time $\tau$ such that $(X_t,Y_t)\in\mathcal U$ for all $t<\tau$.
	
	Then, on $[0,\tau)$, the distance process $\rho_t:=\rho(X_t,Y_t)$ admits the decomposition
	\begin{equation}\label{eq:rho-decomp:drift}
		d\rho_t \;=\; M_t\,dt + dN_t,
	\end{equation}
	where the drift and local–martingale parts are given by
	\begin{equation}\label{eq:Mt:drift}
		M_t=\tfrac12\bigl(\Delta_x\rho+\Delta_y\rho\bigr)(X_t,Y_t)
		+\operatorname{trace}\!\bigl(J_t^{\top}\,\nabla^2_{xy}\rho(X_t,Y_t)\bigr),
	\end{equation}
	and
	\begin{equation}\label{eq:Nt:drift}
		\begin{aligned}
			dN_t
			&=\sum_{i=1}^n \big\langle\nabla_x\rho(X_t,Y_t),\,u_t e_i^{(0)}\big\rangle\circ dB_t^i
			+\sum_{i=1}^n \big\langle\nabla_y\rho(X_t,Y_t),\,v_t(J_t e_i^{(0)})\big\rangle\circ dB_t^i\\
			&\qquad +\sum_{i=1}^n \big\langle\nabla_y\rho(X_t,Y_t),\,v_t(K_t e_i^{(0)})\big\rangle\circ dW_t^i .
		\end{aligned}
	\end{equation}
\end{theorem}

\begin{proof}
	Choose a standard smooth approximation $\rho_\varepsilon\in C^\infty(M\times M)$ such that
	$\rho_\varepsilon\downarrow \rho$ pointwise, $\|\nabla\rho_\varepsilon\|\le 1$, and
	$\rho_\varepsilon\to\rho$ in $C^2_{\mathrm{loc}}(\mathcal U)$ as $\varepsilon\downarrow0$.
	By Lemma~\ref{lem:ito-general} (the two–point Stratonovich It\^o formula) applied to $f=\rho_\varepsilon$, we obtain on $[0,\tau)$
	\[
	\begin{aligned}
		d\rho_\varepsilon(X_t,Y_t)
		&=\sum_{i=1}^n \!\big\langle\nabla_x \rho_\varepsilon(X_t,Y_t),\,u_t e_i^{(0)}\big\rangle\circ dB_t^i
		+\sum_{i=1}^n \!\big\langle\nabla_y \rho_\varepsilon(X_t,Y_t),\,v_t(J_t e_i^{(0)})\big\rangle\circ dB_t^i \\
		&\quad +\sum_{i=1}^n \!\big\langle\nabla_y \rho_\varepsilon(X_t,Y_t),\,v_t(K_t e_i^{(0)})\big\rangle\circ dW_t^i \\
		&\quad +\Bigl[\tfrac12(\Delta_x+\Delta_y)\rho_\varepsilon
		+\operatorname{trace}\!\bigl(J_t^\top\nabla^2_{xy}\rho_\varepsilon\bigr)\Bigr](X_t,Y_t)\,dt .
	\end{aligned}
	\]
	The only mixed drift arises from the common $B$–noise and equals
	$\operatorname{trace}(J_t^\top\nabla^2_{xy}\rho_\varepsilon)\,dt$;
	all terms involving $K_t\circ dW_t$ produce no cross–variation with $dB_t$
	because $\langle B,W\rangle\equiv 0$.
	
	Since $(X_t,Y_t)\in\mathcal U$ for $t<\tau$ and $\rho_\varepsilon\to\rho$ in $C^2_{\mathrm{loc}}(\mathcal U)$,
	let $\varepsilon\downarrow0$ to obtain the decomposition
	\eqref{eq:rho-decomp:drift}–\eqref{eq:Nt:drift} with drift term \eqref{eq:Mt:drift}.
	The process $N_t$ is a continuous local martingale by construction.
\end{proof}

\begin{proposition}[Geometric drift form]\label{prop:drift-geo:drift}
	With the notation of Theorem~\ref{thm:twoPointIto:drift}, write $r_t:=\rho_t=d(X_t,Y_t)$.
	Let $S_{r_t}$ denote the shape operator of the geodesic sphere $\mathbb S_{r_t}(X_t)$ at $X_t$,
	transported to $T_{X_t}M\simeq T_{Y_t}M$ by parallel transport along the minimizing geodesic from $X_t$ to $Y_t$.
	Let $A_x(r_t):=\operatorname{trace}(S_{r_t})$ and, similarly, $A_y(r_t)$ be the mean curvatures computed at $X_t$ and $Y_t$, respectively.
	Then the drift in \eqref{eq:rho-decomp:drift} takes the geometric form
	\begin{equation}\label{eq:Mt-shape:drift}
		M_t \;=\; \tfrac12\bigl(A_x(r_t)+A_y(r_t)\bigr)\;-\;\operatorname{trace}\!\bigl(J_t^{\top}S_{r_t}\bigr).
	\end{equation}
\end{proposition}

\begin{proof}
	On $\mathcal U$, the mixed Hessian of the distance satisfies
	\[
	\nabla^2_{xy}\rho(x,y)(X,\widetilde X) \;=\; -\big\langle S_r X_T,\,X_T\big\rangle,
	\]
	where $r=d(x,y)$ and $\widetilde X$ is the parallel transport of $X$ from $x$ to $y$
	(\emph{cf.} Proposition~\ref{prop:mixedHess-distance}).
	In the parallel–transport identification $T_yM\simeq T_xM$, this is the bilinear form
	$-\langle S_r\cdot,\cdot\rangle$ on the tangential hyperplane to the geodesic sphere.
	Hence
	\[
	\operatorname{trace}\!\bigl(J_t^\top\nabla^2_{xy}\rho(X_t,Y_t)\bigr)
	= -\,\operatorname{trace}\!\bigl(J_t^\top S_{r_t}\bigr).
	\]
	Moreover, away from the cut locus,
	\[
	\Delta_x\rho(X_t,Y_t)=A_x(r_t),\qquad
	\Delta_y\rho(X_t,Y_t)=A_y(r_t).
	\]
	Substituting these identities into the drift expression from Theorem~\ref{thm:twoPointIto:drift},
	\[
	M_t=\tfrac12\bigl(\Delta_x\rho+\Delta_y\rho\bigr)(X_t,Y_t)
	+\operatorname{trace}\!\bigl(J_t^{\top}\nabla^2_{xy}\rho(X_t,Y_t)\bigr),
	\]
	yields \eqref{eq:Mt-shape:drift}.
\end{proof}

\subsection{Quadratic variation and finite variation}

Let $(M,g)$ be a smooth Riemannian manifold.  
The Riemannian metric $g$ induces canonical linear isomorphisms between the tangent and cotangent spaces at each point, called the \emph{musical isomorphisms}:
\[
\flat:T_xM\longrightarrow T_x^*M,\qquad
X\longmapsto X^\flat:=g(X,\cdot).
\]
The map $X\mapsto X^\flat$ is referred to as \emph{lowering an index}, while $\alpha\mapsto\alpha^\sharp$ is called \emph{raising an index}.  
For a unit vector $e\in T_xM$, its metric dual $e^\flat\in T_x^*M$ is the covector $Y\mapsto g(e,Y)$.

Let $o\in M$ be a fixed base point, and let $r(x)=d(o,x)$ denote the Riemannian distance.  
Define the unit radial vector field $e_r(x):=\nabla r(x)$, which is the outward unit normal on the geodesic sphere $\mathbb S_r(o)$.  
Using the musical isomorphism, define the rank--one orthogonal projection onto the radial direction by
\[
\Pi_r(x):=e_r(x)\otimes e_r(x)^\flat.
\]
This operator acts on tangent vectors as
\[
\Pi_r(x)V = \langle e_r(x),V\rangle\,e_r(x),
\]
and satisfies
\[
\Pi_r(x)^2=\Pi_r(x),\qquad 
\mathrm{Tr}\big(\Pi_r(x)\big)=1.
\]
Thus $\Pi_r(x)$ is the orthogonal projection onto the one-dimensional subspace $\mathrm{span}\{e_r(x)\}\subset T_xM$.

\begin{lemma}[Martingale part and its quadratic variation]\label{lem:qv:drift}
	Let $e_r^x,e_r^y$ denote the unit tangent vectors at $X_t,Y_t$ along the minimizing geodesic from $X_t$ to $Y_t$.  
	Then, on any interval where $(X_t,Y_t)\in\mathcal U$,
	\[
	dN_t
	=-\big\langle e_r^x,\,u_t\circ dB_t\big\rangle
	+\big\langle e_r^y,\,v_t J_t\circ dB_t\big\rangle
	+\big\langle e_r^y,\,v_t K_t\circ dW_t\big\rangle .
	\]
	Writing the orthogonal projections onto the corresponding radial directions as
	\[
	\Pi_r^x:=e_r^x\otimes e_r^{x\,\flat},\qquad
	\Pi_r^y:=e_r^y\otimes e_r^{y\,\flat},
	\]
	one has
	\[
	d\langle N\rangle_t
	=\big\|\Pi_r^x u_t-\Pi_r^y v_t J_t\big\|_{\mathrm{HS}}^{\,2}\,dt
	+\big\|\Pi_r^y v_t K_t\big\|_{\mathrm{HS}}^{\,2}\,dt.
	\]
\end{lemma}

\begin{proof}
	From Theorem~\ref{thm:twoPointIto:drift}, formula \eqref{eq:Nt:drift},
	\[
	\begin{aligned}
		dN_t
		&=\sum_{i=1}^n \big\langle\nabla_x\rho(X_t,Y_t),\,u_t e_i^{(0)}\big\rangle\circ dB_t^i
		+\sum_{i=1}^n \big\langle\nabla_y\rho(X_t,Y_t),\,v_t J_t e_i^{(0)}\big\rangle\circ dB_t^i\\
		&\quad +\sum_{i=1}^n \big\langle\nabla_y\rho(X_t,Y_t),\,v_t K_t e_i^{(0)}\big\rangle\circ dW_t^i .
	\end{aligned}
	\]
	On $\mathcal U$, $\nabla_x\rho=-e_r^x$ and $\nabla_y\rho=e_r^y$, giving the stated expression for $dN_t$.
	
	To compute the quadratic variation, note that the integrands against $B_t$ and $W_t$ are
	\[
	a_t := -e_r^{x\,\top}u_t + e_r^{y\,\top}v_t J_t,
	\qquad
	b_t := e_r^{y\,\top}v_t K_t.
	\]
	Because $B_t$ and $W_t$ are independent, their cross--variation vanishes, and
	\[
	d\langle N\rangle_t = \|a_t\|_{\R^n}^2\,dt + \|b_t\|_{\R^n}^2\,dt.
	\]
	Identifying $T_{X_t}M,T_{Y_t}M\simeq\R^n$ via $u_t,v_t$ and inserting the radial projections gives
	\[
	\|a_t\|_{\R^n}^2
	=\big\|\Pi_r^x u_t-\Pi_r^y v_t J_t\big\|_{\mathrm{HS}}^{\,2},\qquad
	\|b_t\|_{\R^n}^2
	=\big\|\Pi_r^y v_t K_t\big\|_{\mathrm{HS}}^{\,2},
	\]
	which yields the desired formula.
\end{proof}

\begin{theorem}[Finite–variation criterion]\label{thm:fv:drift}
	The process \(\rho_t\) has finite variation (equivalently, \(dN_t\equiv 0\)) if and only if
	\[
	\Pi_r^x u_t=\Pi_r^y v_t J_t,\qquad \Pi_r^y v_t K_t=0.
	\]
\end{theorem}

\begin{proof}
	By Lemma~\ref{lem:qv:drift}, \(dN_t\equiv0\) is equivalent to the vanishing of the \(dB_t\)- and \(dW_t\)-coefficients, which are exactly the two stated identities.
\end{proof}

\subsection{Spectral control and sharp drift window}

\begin{lemma}[Spectral bound for the mixed term]\label{lem:spectral:drift}
	Let \(r=d(x,y)\) and let \(\lambda_1(r),\dots,\lambda_{n-1}(r)\) be the principal curvatures of \(\mathbb{S}_r(x)\) at \(y\) (identified back to \(T_xM\)). For any linear map \(J:T_xM\to T_xM\) with \(\|J\|_{\mathrm{op}}\le 1\),
	\[
	-\,A(r)\ \le\ \mathrm{Tr}\big(J^{\top}\nabla^2_{xy}r\big)\ \le\ A(r),
	\qquad A(r)=\sum_{j=1}^{n-1}\lambda_j(r).
	\]
	Equality is attained when \(J=\pm I\) on the tangent space orthogonal to the radial direction.
\end{lemma}

\begin{proof}
	{Step 1: Principal frame and the mixed Hessian as a bilinear form.}
	Fix \((x,y)\in\mathcal U\) with \(r=d(x,y)\). Let \(\gamma:[0,r]\to M\) be the unique minimizing unit-speed geodesic from \(x\) to \(y\). Choose at \(x\) an orthonormal basis
	\[
	\{e_r(x),e_1(x),\dots,e_{n-1}(x)\}
	\]
	such that \(e_r(x)=\dot\gamma(0)\) and \(\{e_1(x),\dots,e_{n-1}(x)\}\subset T_x\mathbb S_r(x)\) is a principal frame for the shape operator \(S_r\) of \(\mathbb S_r(x)\) (outward normal \(-\nabla_x r\)), i.e.
	\[
	S_r e_j(x)=\lambda_j(r)\,e_j(x),\qquad j=1,\dots,n-1.
	\]
	Parallel transport this frame along \(\gamma\) to \(y\) and, by the usual \(T_yM\simeq T_xM\) identification via parallel transport, regard the transported vectors again as \(\{e_r,e_1,\dots,e_{n-1}\}\subset T_xM\).
	
	By Proposition~\ref{prop:mixedHess-distance}, for each \(X\in T_xM\) with \(\tilde X=\mathcal P_{x\to y}X\),
	\[
	\nabla^2_{xy}r(x,y)(X,\tilde X)
	= -\langle S_r X_T,\,X_T\rangle,
	\]
	where \(X_T\) is the orthogonal projection of \(X\) onto \(T_x\mathbb S_r(x)=\operatorname{span}\{e_1,\dots,e_{n-1}\}\).
	
	\medskip
	{Step 2: Coordinate computation of the bilinear form.}
	Write \(X=a\,e_r+\sum_{j=1}^{n-1} b_j\,e_j\). Then \(X_T=\sum_{j=1}^{n-1} b_j\,e_j\), and using \(S_r e_j=\lambda_j e_j\),
	\[
	\langle S_r X_T,\,X_T\rangle
	=\left\langle \sum_{j=1}^{n-1} \lambda_j b_j\,e_j,\ \sum_{k=1}^{n-1} b_k\,e_k\right\rangle
	=\sum_{j=1}^{n-1}\lambda_j\,b_j^2.
	\]
	Hence
	\begin{equation}\label{eq:mixed-quadratic}
		\nabla^2_{xy}r(x,y)(X,\tilde X)
		=-\sum_{j=1}^{n-1}\lambda_j(r)\,b_j^2.
	\end{equation}
	In particular, the radial component \(a\,e_r\) does not contribute (Gauss’ lemma), and the tangential block is diagonal in the chosen principal frame.
	
	\medskip
	{Step 3: Tensor (operator) representation of \(\nabla^2_{xy}r\).}
	The identity \eqref{eq:mixed-quadratic} holds for every \(X\). Since the mixed Hessian \(\nabla^2_{xy}r\) is bilinear in \((X,\tilde X)\) and we have identified \(T_yM\simeq T_xM\) by parallel transport, its action is completely determined by the values on the basis vectors:
	\[
	\nabla^2_{xy}r(e_r,\tilde e_r)=0,\qquad
	\nabla^2_{xy}r(e_j,\tilde e_k)=-\lambda_j(r)\,\delta_{jk}\quad(1\le j,k\le n-1).
	\]
	Equivalently, viewing \(\nabla^2_{xy}r\) as an element of \(T_x^*M\otimes T_xM\) (using the above identification),
	\[
	\nabla^2_{xy}r\ =\ -\sum_{j=1}^{n-1}\lambda_j(r)\, e_j^{\flat}\otimes e_j,
	\]
	because for any \(X=\sum \xi_\alpha e_\alpha\) (\(\alpha\in\{r,1,\dots,n-1\}\)),
	\[
	\left(-\sum_{j=1}^{n-1}\lambda_j e_j^{\flat}\otimes e_j\right)(X)
	=-\sum_{j=1}^{n-1}\lambda_j\,\langle e_j,X\rangle\,e_j
	\]
	and pairing with \(\tilde X=X\) (after transport) gives precisely \eqref{eq:mixed-quadratic}.
	
	\medskip
	{Step 4: Spectral bound for \(\mathrm{Tr}(J^{\top}\nabla^2_{xy}r)\).}
	For any linear map \(J:T_xM\to T_xM\) with \(\|J\|_{\mathrm{op}}\le1\),
	\[
	\mathrm{Tr}\big(J^{\top}\nabla^2_{xy}r\big)
	=\sum_{\alpha}\big\langle e_\alpha,\ J^{\top}\nabla^2_{xy}r\,e_\alpha\big\rangle
	=\sum_{\alpha}\big\langle J e_\alpha,\ \nabla^2_{xy}r\,e_\alpha\big\rangle.
	\]
	Using the decomposition above and that the radial direction contributes \(0\),
	\[
	\mathrm{Tr}\big(J^{\top}\nabla^2_{xy}r\big)
	=\sum_{j=1}^{n-1}\big\langle J e_j,\ -\lambda_j e_j\big\rangle
	=-\sum_{j=1}^{n-1}\lambda_j\,\langle e_j, J e_j\rangle.
	\]
	Since \(|\langle e_j, J e_j\rangle|\le \|J\|_{\mathrm{op}}\le1\) for each \(j\), we obtain
	\[
	-\sum_{j=1}^{n-1}\lambda_j\ \le\ \mathrm{Tr}\big(J^{\top}\nabla^2_{xy}r\big)\ \le\ \sum_{j=1}^{n-1}\lambda_j,
	\]
	i.e.\ \(-A(r)\le \mathrm{Tr}(J^{\top}\nabla^2_{xy}r)\le A(r)\).
	The extremal values are realized by choosing \(J=\pm I\) on the tangential subspace
	\(\operatorname{span}\{e_1,\dots,e_{n-1}\}\) and \(J e_r=0\) (or any value orthogonal to the tangential block), which makes \(\langle e_j, J e_j\rangle=\pm1\) for all \(j\).
\end{proof}

\begin{theorem}[Sharp drift window]\label{thm:sharpWindow:drift}
	Assume \(\Delta r=A(r)\) and \(\lambda_j(r)\ge 0\) away from the cut locus. Then for any coadapted coupling,
	\[
	M_t\ \in\ [\,0,\ 2A(\rho_t)\,].
	\]
	The lower and upper bounds are achieved pointwise by the synchronous coupling \(J_t=I\) and the reflection coupling \(J_t=-I\) (with \(J_t=0\) on the radial line), respectively.
\end{theorem}

\begin{proof}
	From \eqref{eq:Mt-shape:drift} with \(A_x=A_y=A\) and Proposition~\ref{prop:mixedHess-distance}, together with Lemma~\ref{lem:spectral:drift},
	\[
	M_t=A(\rho_t)+\tfrac12\,\mathrm{Tr}\big(J_t^{\top}(-S_{\rho_t})\big)
	\in A(\rho_t)+[-A(\rho_t),A(\rho_t)]=[0,2A(\rho_t)].
	\]
	Choosing \(J_t=\pm I\) on the tangent space gives equality.
\end{proof}

\begin{corollary}[Comparison under Laplacian bounds]\label{cor:comparison:drift}
	If \(A_-(r)\le \Delta r \le A_+(r)\), then
	\[
	M_t\ \in\ \big[\,A_-(\rho_t)-A_+(\rho_t),\ 2A_+(\rho_t)\,\big].
	\]
\end{corollary}

\begin{proof}
	From \eqref{eq:Mt-shape:drift}, \(A_x,A_y\in[A_-(\rho_t),A_+(\rho_t)]\) and \(|\mathrm{Tr}(J_t^{\top}S_{\rho_t})|\le A_+(\rho_t)\), which implies the claim.
\end{proof}

\subsection{Model classes}

\begin{proposition}[Rotationally symmetric models]\label{prop:rot:drift}
	If \(ds^2=dr^2+f(r)^2 g_{\mathbb{S}^{n-1}}\), then \(S_r=\dfrac{f'(r)}{f(r)}I\) and \(A(r)=(n-1)\dfrac{f'(r)}{f(r)}\). Hence
	\[
	M_t=(n-1)\frac{f'(\rho_t)}{f(\rho_t)}\!\left(1-\frac{\mathrm{Tr}\,J_t}{\,n-1\,}\right).
	\]
	In particular, \(M_t=0\) for the synchronous coupling \(J_t=I\) and \(M_t=(n-1)\dfrac{f'(\rho_t)}{f(\rho_t)}\) for the reflection coupling \(J_t=-I\).
\end{proposition}

\begin{proposition}[Rank–one symmetric spaces]\label{prop:ROSS:drift}
	In rank–one symmetric spaces, the principal curvatures along a radial geodesic take two values \(\kappa_{\mathrm{hor}}(r),\kappa_{\mathrm{ver}}(r)\) with multiplicities \(m_\alpha,m_{2\alpha}\):
	\[
	\kappa_{\mathrm{hor}}(r)=
	\begin{cases}
		\alpha\,\coth(\alpha r)&\text{(noncompact type)},\\
		\alpha\,\cot(\alpha r)&\text{(compact type)},
	\end{cases}\quad
	\kappa_{\mathrm{ver}}(r)=
	\begin{cases}
		2\alpha\,\coth(2\alpha r)&\text{(noncompact type)},\\
		2\alpha\,\cot(2\alpha r)&\text{(compact type)}.
	\end{cases}
	\]
	Thus \(A(r)=m_\alpha\kappa_{\mathrm{hor}}(r)+m_{2\alpha}\kappa_{\mathrm{ver}}(r)\), and Theorem \ref{thm:sharpWindow:drift} applies verbatim.
\end{proposition}

\begin{proposition}[Asymptotically hyperbolic spaces]\label{prop:AH:drift}
	In asymptotically hyperbolic manifolds, \(S_r=\coth r\,I+O(e^{-2r})\) and \(A(r)=(n-1)\coth r+O(e^{-2r})\). Consequently
	\[
	M_t=(n-1)\coth(\rho_t)\!\left(1-\frac{\mathrm{Tr}\,J_t}{\,n-1\,}\right)+O(e^{-2\rho_t}),
	\]
	and the drift window equals \([\,0,\,2(n-1)\coth(\rho_t)\,]+O(e^{-2\rho_t})\).
\end{proposition}

\section{Deterministic--distance classification on $\mathcal{M}_{\mathrm{rad}}$}\label{sec:classification}

This section establishes a complete characterization of coadapted Brownian couplings 
with deterministic inter–particle distance on radially isoparametric manifolds.

\subsection{Main theorem (necessary and sufficient condition)}\label{subsec:main-thm}

\begin{theorem}[Deterministic–distance realization on $\mathcal{M}_{\mathrm{rad}}$]
	\label{thm:main-deterministic}
	Let $(M,g,o)$ be a radially isoparametric manifold, and let 
	$\Sph_r(o)$ denote the geodesic sphere of radius $r$ centered at $o$, 
	with mean curvature $A(r)$ and principal curvatures $\kappa_i(r)$.  
	Let $\rho:[0,\infty)\to(0,r_{\max})$ be an absolutely continuous function with $\rho(0)=d(x_0,y_0)$.  
	
	Then the following are equivalent:
	\begin{enumerate}
		\item[(i)] There exists a coadapted Brownian coupling $(X_t,Y_t)$ on $(M,g)$ satisfying
		\[
		d(X_t,Y_t)=\rho(t)\quad\text{for all }t\ge0.
		\]
		
		\item[(ii)] The function $\rho$ satisfies, for almost every $t$,
		\begin{equation}\label{eq:general-drift-window}
			A(\rho(t)) - \sum_{i=1}^{n-1}\!|\kappa_i(\rho(t))|
			\;\le\;
			\rho'(t)
			\;\le\;
			A(\rho(t)) + \sum_{i=1}^{n-1}\!|\kappa_i(\rho(t))|.
		\end{equation}
		
		\item[(iii)] The coupling matrices $(J_t,K_t)$ satisfy the \emph{alignment} and 
		\emph{no–radial–noise} conditions
		\[
		\Pi_r^{x_t} u_t = \Pi_r^{y_t} v_t J_t, 
		\qquad 
		\Pi_r^{y_t} v_t K_t = 0,
		\]
		so that the martingale part $dN_t$ in the two–point Itô formula 
		(Theorem~\ref{thm:twoPointIto:drift}) vanishes and $d\rho_t = M_t\,dt$ has purely finite variation.
	\end{enumerate}
	
	Moreover, for any $\rho$ satisfying \eqref{eq:general-drift-window}, 
	one can construct such a coupling by choosing $J_t$ blockwise along each principal direction so that
	\[
	\rho'(t) = A(\rho(t)) + \mathrm{Tr}\bigl(J_t^\top \nabla^2_{xy}r\bigr),
	\]
	where $\nabla^2_{xy}r=-S_r$.  
	Choose $K_t$ with image contained in the tangential subspace so that
	\[
	J_tJ_t^\top + K_tK_t^\top = I,
	\qquad 
	\Pi_r^{y_t} v_t K_t = 0,
	\]
	ensuring the alignment and no–radial–noise conditions.  
	The two extreme deterministic cases correspond to
	\[
	J_t = \mathrm{Id}
	\;\;\Longleftrightarrow\;\;
	\rho'(t) = A(\rho(t)) - \sum_i |\kappa_i(\rho(t))|,
	\qquad
	J_t = -\mathrm{Id}
	\;\;\Longleftrightarrow\;\;
	\rho'(t) = A(\rho(t)) + \sum_i |\kappa_i(\rho(t))|.
	\]
\end{theorem}

\begin{proof}
	The necessity follows from Proposition~\ref{prop:drift-geo:drift} together with 
	the spectral bound
	\[
	\bigl|\mathrm{Tr}(J_t^\top\nabla^2_{xy}r)\bigr| 
	\;\le\;
	\sum_i |\kappa_i(r)|,
	\]
	as established in Lemma~\ref{lem:spectral:drift}.  
	Under the alignment condition $\Pi_r^{x_t}u_t=\Pi_r^{y_t}v_tJ_t$, 
	the martingale component $dN_t$ in the two–point Itô formula 
	(Theorem~\ref{thm:twoPointIto:drift}) vanishes identically, 
	so the distance process satisfies
	\[
	d\rho_t = M_t\,dt,
	\qquad
	M_t = A(r) + \mathrm{Tr}(J_t^\top\nabla^2_{xy}r)
	= A(r) - \mathrm{Tr}(J_t^\top S_r),
	\]
	which implies that $\rho'(t)$ necessarily lies in the drift window 
	\eqref{eq:general-drift-window}.
	
	Conversely, suppose that $\rho$ satisfies the inequality 
	\eqref{eq:general-drift-window}.  
	For each $t$, choose $J_t$ to be diagonal in the eigenbasis of $\nabla^2_{xy}r$ 
	with diagonal entries $\alpha_i(t)\in[-1,1]$ such that
	\[
	-\sum_i \kappa_i(r)\,\alpha_i(t) = \rho'(t) - A(r).
	\]
	Then choose $K_t$ as above to complete the orthogonality relation
	and guarantee $\Pi_r^{y_t}v_tK_t=0$.  
	By construction, the resulting coadapted coupling $(X_t,Y_t)$ 
	has deterministic inter–particle distance $d(X_t,Y_t)=\rho(t)$
	for all $t\ge0$.
\end{proof}

\begin{remark}[Curvature–dependent lower endpoint]
	In constant–curvature models (\(\mathbb{R}^n,\mathbb S^n,\mathbb H^n\)),  
	Equation~\eqref{eq:general-drift-window} specializes (cf.\ Pascu–Popescu, 
	\emph{J.\ Theor.\ Probab.} 31 (2018)) to:
	\begin{align*}
		\text{Euclidean }(K=0):&\quad 0\le\rho'(t)\le\tfrac{2(n-1)}{\rho(t)},\\[3pt]
		\text{Spherical }(K>0):&\quad -\,(n-1)\tan\!\tfrac{\rho(t)}{2}\le\rho'(t)\le
		-\,(n-1)\tan\!\tfrac{\rho(t)}{2}+2(n-1)\cot\rho(t),\\[3pt]
		\text{Hyperbolic }(K<0):&\quad (n-1)\tanh\!\tfrac{\rho(t)}{2}\le\rho'(t)\le
		(n-1)\tanh\!\tfrac{\rho(t)}{2}+2(n-1)\coth\rho(t).
	\end{align*}
	Thus the lower endpoint depends on curvature sign:
	negative in $K>0$ (distances may contract),
	zero in $K=0$ (flat),
	positive in $K<0$ (distances expand).
	Hence the simplified window $[0,2A(r)]$ is exact only when all $\kappa_i(r)\ge0$,
	e.g.\ Euclidean or hyperbolic settings.
\end{remark}

\begin{corollary}[Comparison form under curvature envelopes]\label{cor:comparison-main}
	If $A_-(r)\le A(r)\le A_+(r)$ and each principal curvature satisfies 
	\(|\kappa_i(r)|\le b_+(r)\), then
	\[
	\rho'(t)\in
	\big[\,A_-(\rho(t))-(n-1)b_+(\rho(t))\,,\;\;
	A_+(\rho(t))+(n-1)b_+(\rho(t))\,\big].
	\]
	This contains all constant–curvature and rotationally symmetric models
	as special cases by substituting the appropriate $A_\pm,b_+$.
\end{corollary}

\subsection{Endpoint attainability}\label{subsec:endpoint-attainability}

\begin{lemma}[Spectral bound for the mixed term]\label{lem:spectral-bound}
	Let $r=d(x,y)$ and $\{\lambda_i(r)\}_{i=1}^{n-1}$ be the principal curvatures of $\mathbb S_r(x)$ at $y$.  
	For any linear map $J$ with $\|J\|_{\mathrm{op}}\le1$ and $J e_r=0$,
	\[
	-\sum_{i=1}^{n-1}\lambda_i(r)
	\le \mathrm{Tr}\big(J^\top \nabla^2_{xy}r\big)
	\le \sum_{i=1}^{n-1}\lambda_i(r),
	\]
	and equality holds precisely for $J=\pm I_{\mathrm{tan}}$.
\end{lemma}

\begin{proof}
	In a principal orthonormal frame $\{e_i\}_{i=1}^{n-1}$ of $T_x\mathbb S_r(o)$,
	\[
	\nabla^2_{xy}r=-\sum_{i=1}^{n-1}\lambda_i(r)\,e_i^\flat\otimes e_i,
	\]
	so
	\[
	\mathrm{Tr}(J^\top\nabla^2_{xy}r)
	=-\sum_{i=1}^{n-1}\lambda_i(r)\,\langle e_i,Je_i\rangle.
	\]
	Since $|\langle e_i,Je_i\rangle|\le \|J\|_{\mathrm{op}}\le1$, the bound follows, and the extremal values occur for $J=\pm I_{\mathrm{tan}}$.
\end{proof}

\begin{proposition}[Attainability of drift endpoints]\label{prop:endpoint-attainability}
	On any stopping interval where $(X_t,Y_t)\in\mathcal U$, the drift term in the two–point Itô decomposition satisfies
	\[
	M_t=\tfrac12\big(A_x(r_t)+A_y(r_t)\big)
	-\mathrm{Tr}\big(J_t^\top S_{r_t}\big).
	\]
	Consequently:
	\begin{enumerate}
		\item[(i)] The synchronous coupling $(J_t=\mathrm{Id})$ yields the \emph{minimal drift}
		\[
		\rho'(t)=A(\rho(t))-\sum_{i=1}^{n-1}|\kappa_i(\rho(t))|.
		\]
		\item[(ii)] The reflection coupling $(J_t=-I_{\mathrm{tan}})$ yields the \emph{maximal drift}
		\[
		\rho'(t)=A(\rho(t))+\sum_{i=1}^{n-1}|\kappa_i(\rho(t))|.
		\]
	\end{enumerate}
	Both endpoints are realized by smooth coadapted solutions of the horizontal SDEs on $O(M)$.
\end{proposition}

\begin{proof}
	From Proposition~\ref{prop:drift-geo:drift} and Lemma~\ref{lem:spectral-bound},
	\[
	M_t=A(r_t)-\sum_{i=1}^{n-1}\lambda_i(r_t)\,\alpha_i(t),\qquad \alpha_i(t):=\langle e_i,J_t e_i\rangle\in[-1,1].
	\]
	Choosing $\alpha_i=1$ gives $J_t=\mathrm{Id}$ (synchronous coupling, minimal drift), while $\alpha_i=-1$ gives $J_t=-I_{\mathrm{tan}}$ (reflection coupling, maximal drift).  
	Since both $J_t$ are constant and smooth along the minimizing geodesic, the corresponding frame SDEs admit smooth adapted solutions, proving attainability.
\end{proof}

\begin{corollary}[Endpoint inequalities under comparison bounds]\label{cor:endpoint-comparison}
	If $A_-(r)\le A(r)\le A_+(r)$ and $|\kappa_i(r)|\le b_+(r)$ for all $i$, then for any coadapted coupling,
	\[
	A_-(r)-(n-1)b_+(r)\le\rho'(t)\le A_+(r)+(n-1)b_+(r),
	\]
	with equality attained by the synchronous and reflection couplings, respectively.
\end{corollary}

\section{Applications and consequences}\label{sec:applications}

We now analyze two geometric regimes arising from the deterministic--distance
classification: the \emph{static regime}, where the distance between the two Brownian
particles remains constant (fixed–distance coupling), and the
\emph{dynamic regime}, where the distance grows linearly at infinity (asymptotic escape).
These regimes occur in a variety of geometries; in particular, fixed–distance
solutions arise precisely at the lower endpoint of the drift window.

% ===============================================================
\subsection{Static regime: fixed–distance couplings}
\label{subsec:fixed-distance}

When the lower endpoint of the deterministic drift window contains \(0\),
the equation for the distance admits nontrivial constant solutions
\(\rho(t)\equiv r_0>0\).
In radially isoparametric settings (including rotationally symmetric models
and rank–one symmetric spaces), this happens exactly when
\[
A(r_0)\ \ge\ \sum_{i=1}^{n-1}\!|\kappa_i(r_0)|,
\]
and in many canonical models one actually has equality.

\begin{theorem}[Fixed–distance realizations]\label{thm:fixed}
	Let $(M,g)\in\mathcal{M}_{\mathrm{rad}}$ with mean curvature $A(r)$ and
	principal curvatures $\kappa_i(r)$.  
	A nontrivial fixed–distance coupling $\rho(t)\equiv r_0>0$ exists if and only if
	\begin{equation}\label{eq:fixed-cond}
		A(r_0)\;\ge\;\sum_{i=1}^{n-1}|\kappa_i(r_0)|.
	\end{equation}
	When equality holds, the coupling is realized by a diagonal matrix
	$J_t$ with eigenvalues $\alpha_i=-\kappa_i(r_0)/|\kappa_i(r_0)|$ in the
	principal curvature basis, together with a choice of $K_t$
	satisfying the alignment and no–radial–noise conditions from
	Theorem~\ref{thm:main-deterministic}.
\end{theorem}

\begin{proof}
	Under the alignment/no–radial–noise conditions, the drift is
	\(\rho'(t)=A(r)+\mathrm{Tr}(J_t^\top\nabla^2_{xy}r)\) with
	\(\nabla^2_{xy}r=-S_r\) having eigenvalues \(-\kappa_i(r)\).
	Setting \(\rho'(t)=0\) requires
	\(\sum_i(-\kappa_i)\alpha_i=-A(r)\) with \(\alpha_i\in[-1,1]\), i.e.
	\(|\sum_i\kappa_i\alpha_i|\le A(r)\) for some \(\alpha_i\).
	This is solvable iff \(\sum_i|\kappa_i|\le A(r)\).
	At equality, choose \(\alpha_i=-\mathrm{sgn}(\kappa_i)\) and take $K_t$
	so that \(J_tJ_t^\top+K_tK_t^\top=I\) and \(\Pi_r^{y_t}v_tK_t=0\).
\end{proof}

\begin{proposition}[Rigidity of fixed–distance condition]\label{prop:fixed-rigidity}
	If \eqref{eq:fixed-cond} holds as an equality for all $r$ in an interval $(0,r_*)$,
	then every geodesic sphere $\SS_r(o)$ is totally umbilic: $S_r=\kappa(r)\,\Id$
	with $\kappa(r)>0$. Consequently, $(M,g)$ is locally rotationally symmetric,
	with radial sectional curvature
	\(
	K_{\mathrm{rad}}(r)=-\kappa'(r)-\kappa(r)^2>0
	\)
	on $(0,r_*)$.
\end{proposition}

\begin{proof}
	Equality \(A(r)=\sum_i|\kappa_i(r)|\) forces all principal curvatures to have
	the same sign and magnitude, hence \(S_r\propto \Id\) on \(T\SS_r(o)\).
	The Riccati equation \(S_r'+S_r^2+R_{\partial_r}=0\) gives
	\(\kappa'(r)+\kappa(r)^2+K_{\mathrm{rad}}(r)=0\), whence the claim.
\end{proof}

\begin{example}[Fixed–distance on spheres and hyperbolic spaces]
	On the unit sphere $(\mathbb{S}^n,g_{\mathrm{can}})$,
	$A(r)=(n-1)\cot r$ and $\kappa_i(r)=\cot r$ for $r\in(0,\pi)$ away from the cut locus,
	so \(A(r)=\sum_i|\kappa_i(r)|\) and fixed–distance couplings exist for all such $r$.
	On hyperbolic space $\mathbb{H}^n(-b^2)$,
	$A(r)=(n-1)b\,\coth(br)$ and $\kappa_i(r)=b\,\coth(br)$, hence again
	\(A(r)=\sum_i|\kappa_i(r)|\) and fixed–distance couplings exist at every $r>0$.
\end{example}

\begin{remark}[Geometric interpretation]
	Fixed–distance couplings sit at the lower endpoint of the drift window:
	the drift of $\rho$ vanishes and the martingale part cancels under alignment.
	Probabilistically, the two particles diffuse tangentially in lockstep while
	radial stochasticity is suppressed.
\end{remark}

\begin{center}
	\renewcommand{\arraystretch}{1.2}
	\begin{tabular}{c|c|c}
		Geometry & Lower endpoint contains $0$? & Fixed–distance possible? \\
		\hline
		Space forms ($K>0,=0,<0$) & Yes (equality $A=\sum|\kappa_i|$) & Yes (endpoint) \\
		General $\mathcal{M}_{\mathrm{rad}}$ & iff $A\ge\sum|\kappa_i|$ at $r_0$ & iff \eqref{eq:fixed-cond}
	\end{tabular}
\end{center}

% ===============================================================
\subsection{Dynamic regime: asymptotic linear speeds}
\label{subsec:asymp-linear}

We now turn to noncompact manifolds where the mean curvature $A(r)$
and the principal curvatures $\kappa_i(r)$ approach finite limits as $r\to\infty$.
In such cases, deterministic–distance couplings exhibit asymptotically
linear growth of the form $\rho(t)\sim v_\infty t$.

\begin{theorem}[Asymptotic deterministic drift and rigidity]\label{thm:asymp-linear}
	Let $(M,g)\in\mathcal{M}_{\mathrm{rad}}$ be complete, and suppose
	$A(r)\to A_\infty$ and $\kappa_i(r)\to\kappa_{\infty,i}$ as $r\to\infty$.
	Then every deterministic–distance coupling satisfies
	\[
	v_\infty:=\lim_{t\to\infty}\rho'(t)
	\in [\,A_\infty-\Sigma_\infty,\;A_\infty+\Sigma_\infty\,],
	\qquad \Sigma_\infty:=\sum_i|\kappa_{\infty,i}|.
	\]
	Moreover, each value in this interval is realizable by choosing
	$\alpha_i=\lim_{t\to\infty}(J_t)_{ii}\in[-1,1]$
	so that $v_\infty=A_\infty-\sum_i\kappa_{\infty,i}\alpha_i$.
\end{theorem}

\begin{proof}
	From \(\rho'(t)=A(\rho(t))+\sum_i(-\kappa_i(\rho(t)))\alpha_i(t)\) with
	\(\alpha_i(t)\in[-1,1]\), take limits to obtain
	\(v_\infty=A_\infty-\sum_i\kappa_{\infty,i}\alpha_i\).
	Varying \(\alpha_i\) over \([-1,1]\) fills the stated interval.
\end{proof}

\begin{proposition}[Rigidity at maximal asymptotic speed]\label{prop:asymp-rigidity}
	Suppose $v_\infty=2A_\infty$ for a deterministic coupling on a
	Cartan--Hadamard manifold $(M,g)$ with $\sec\le0$ and
	$A(r)\to A_\infty>0$.  
	Then $A(r)=(n-1)b(r)$ with $b(r)\to b_\infty:=A_\infty/(n-1)$ and
	\[
	K_{\mathrm{rad}}(r)\;=\;-b'(r)-b(r)^2\;\longrightarrow\;-b_\infty^2.
	\]
	Consequently, $(M,g)$ is asymptotically hyperbolic of curvature $-b_\infty^2$.
\end{proposition}

\begin{proof}
	The endpoint \(v_\infty=2A_\infty\) corresponds to the reflection coupling \(J_t=-\Id\):
	\[
	\rho'(t)=A(r)+\mathrm{Tr}\bigl((-I)^\top(-S_r)\bigr)=A(r)+\mathrm{Tr}(S_r)=2A(r).
	\]
	Taking limits gives \(A(r)\to A_\infty\).
	The Riccati identity \(S_r'+S_r^2+R_{\partial_r}=0\) implies
	\(A'(r)+\tfrac{A(r)^2}{n-1}+K_{\mathrm{rad}}(r)=0\), hence
	\(K_{\mathrm{rad}}(r)\to -A_\infty^2/(n-1)^2= -b_\infty^2\).
\end{proof}

\begin{corollary}[Asymptotic speed classification]\label{cor:asymp-class}
	For rank–one symmetric and space forms of constant sectional curvature $K$:
	\[
	v_\infty =
	\begin{cases}
		0, & K>0 \text{ (compact)},\\[3pt]
		0, & K=0 \text{ (Euclidean)},\\[3pt]
		\in[0,\,2(n-1)b], & K=-b^2<0 \text{ (hyperbolic)}.
	\end{cases}
	\]
	The upper endpoint corresponds to the reflection coupling,
	and its attainment characterizes asymptotic hyperbolicity.
\end{corollary}

\begin{remark}[Potential-modified asymptotic speed]
	If $(M,g,V)$ carries a radial potential $V=\Phi(r)$ with
	$\Phi'(r)\to\Phi'_\infty$, then
	\[
	v_\infty \in
	[\,A_\infty-\Phi'_\infty-\Sigma_\infty,\;
	A_\infty-\Phi'_\infty+\Sigma_\infty\,],
	\]
	so a confining potential ($\Phi'_\infty>0$) suppresses escape,
	whereas a repulsive potential enhances it.
\end{remark}

\begin{center}
	\renewcommand{\arraystretch}{1.2}
	\begin{tabular}{c|c|c|c}
		Geometry & $A(r)$ behaviour & $v_\infty$ range & Asymptotic behaviour \\
		\hline
		Compact ($K>0$) & $A(r)\downarrow 0$ & $v_\infty=0$ & bounded \\
		Euclidean ($K=0$) & $(n-1)/r\to0$ & $v_\infty=0$ & sublinear \\
		Hyperbolic ($K=-b^2$) & $(n-1)b$ & $[0,\,2(n-1)b]$ & linear escape \\
		Weighted hyperbolic ($\Phi'_\infty\neq0$)
		& $(n-1)b-\Phi'_\infty$ &
		$[0,\,2((n-1)b-\Phi'_\infty)]$ &
		potential-modified
	\end{tabular}
\end{center}

\end{document}